\documentclass[a4,12pt]{amsart} \topmargin      -15mm \textwidth      160 true mm
\usepackage{amsmath, amssymb, amsfonts, amstext, amsthm}
\usepackage[mathscr]{euscript}
\usepackage{enumerate}
\usepackage[margin=1in]{geometry}
\usepackage{float}
\usepackage{pgf,tikz,ulem}

\newtheorem{thm}{Theorem}[section]
\newtheorem{lem}[thm]{Lemma}
\newtheorem{cor}[thm]{Corollary}
\newtheorem{prop}[thm]{Proposition}
\newtheorem{Exm}[thm]{Example}
\newtheorem{rem}[thm]{Remark}

\def\be{\begin{equation}}
\def\ee{\end{equation}}
\def\bea{\begin{eqnarray}}
\def\eea{\end{eqnarray}}
\begin{document}

\begin{center}
{\bf \large Group Inverses of Weighted Trees}\\

\vspace{3cm}
{\bf Raju Nandi}\\
\vspace{.5cm}
Discipline of Mathematics\\
Indian Institute of Technology Gandhinagar\\
Gandhinagar 382355\\
India \\
\end{center}
  
\vspace{3cm}  
\noindent{\bf Abstract :} \\
Let $(G,w)$ be an undirected weighted graph. The group inverse of $(G,w)$ is the weighted graph with the adjacency matrix $A^{\#}$, where $A$ is the adjacency matrix of $(G,w)$. We study the group inverse of singular weighted trees. It is shown that if $(T,w)$ is a singular weighted tree, then $T^{\#}$ is again a tree, if and only if $T$ is a star tree, which in turn, holds if and only if $T^{\#}$ is graph isomorphic to $T$. A new class $\mathbb{T}$ of weighted trees, is introduced and studied here. It is shown that the group inverse of the adjacency matrix of a positively weighted tree in $\mathbb{T}$, is signature similar to a non-negative matrix.

\vskip.25in
\noindent\textit{Keywords:} 
Weighted graph; adjacency matrix; group inverse of graph; maximum matching; alternating path; star. 

%\vskip.05in
\noindent\textit{AMS Subject Classifications:} 05C22, 05C50, 15A09.

\newpage
\section{Introduction}
Let $A=(a_{ij})$ be an $n\times n$ real symmetric matrix.  Then the {\it undirected weighted graph} of $A$, denoted by $\mathcal{G}(A)$, is the pair $(G,w)$, where $G$ is the graph whose vertices are denoted by $\{v_1,v_2,\ldots,v_n\}$ such that $v_iv_j$ is an edge of $G$ if and only if $a_{ij}\neq 0$ (loops are allowed) and $w$ is the weighting function that assigns to each edge $v_iv_j$ in $G$, the weight $w(v_iv_j)=a_{ij}$. One can also turn this around. Let $(G,w)$ be an undirected weighted graph on $n$ vertices $\{v_1,v_2,\ldots,v_n\}$. Define the {\it adjacency} matrix of $G$ to be the $n\times n$ matrix $A$ defined by $$a_{ij}=\left\{\begin{array}{cl} w(v_iv_j)&\mbox{if }v_iv_j\mbox{ is an edge of }G,\\0&\mbox{otherwise.}\end{array}\right.$$

Note that, for an unweighted undirected graph all weights are assumed to be 1. Let $(G,w)$ be a undirected loop-free weighted graph with adjacency matrix $A$. If $A$ is invertible, then we say that the (weighted) graph $(G^{-1},w^{-1})$ is the {\it inverse} of $(G,w)$. Here, we presume that $G$ and $G^{-1}$ have the same vertex set. Also, $v_iv_j$ is an edge in $G^{-1}$ if, and only if, the $ij$-th entry of $A^{-1}$ is nonzero and the weight $w^{-1}(v_iv_j)$ of the edge $v_iv_j$ equals that entry. It is well known that the inverse defined in this manner is unique up to graph isomorphisms which also preserve the edge weights. In the case of singular matrices, one defines the group inverse of the graph $(G,w)$ in an entirely similar manner, if the group inverse $A^{\#}$ exists. Let us recall that the group inverse of a matrix $A\in \mathbb{R}^{n \times n}$ is the unique matrix $X \in \mathbb{R}^{n \times n}$, if it exists, that satisfies the equations $AXA=A, XAX=X$ and $AX=XA$. Also, for a real rectangular matrix $A$, the Moore-Penrose inverse of $A$, is the unique matrix $A^{\dagger}$ that satisfies the equations $AA^{\dagger}A=A,~A^{\dagger}AA^{\dagger}=A^{\dagger},~(AA^{\dagger})^T=AA^{\dagger}$ and $(A^{\dagger}A)^T=A^{\dagger}A$. In particular, one may show that $A^{\dag}=A^T(AA^T)^{\dag}.$ It is also known that the Moore-Penrose inverse coincides with the group inverse for those square matrices that satisfy the condition $R(A)=R(A^T).$ Here $R(X)$ denotes the range space of the matrix $X$. We refer the reader to \cite{bg} for more details on these notions of generalized inverses and Moore-Penrose inverses.

Thus, the {\it group inverse} of $G$ is the graph  denoted by $(G^\#,w^\#)$ where, as earlier, the vertex sets of $G$ and $G^{\#}$ are taken to be the same, while the weight $w^{\#}$ of an edge $v_iv_j$ in $G^{\#}$ is defined as the $ij$-th entry of $A^{\#}$. Also, as in the case of nonsingular $A$, the group inverse graph $G^{\#}$ is uniquely determined up to isomorphism preserving the edge weights.

Before embarking on proving our results, let us make the following observation, to place this work, in a proper perspective. Irrespective of whether the graph is weighted or otherwise since the adjacency matrix of an undirected graph is always symmetric, (and the group inverse of a symmetric matrix always exists,)  the group inverse of an undirected graph always exists. For directed graphs, very little is known. For instance, we state one result \cite[Proposition 1.1]{cov}: Given a square matrix $A$, let $D(A)$ be a directed tree. Then $A^{\#}$ exists if and only if $D(A)$ satisfies a certain condition involving maximal matchings. 

Before proceeding further, let us give an overview of why the group inverse of a graph may be important to study. Some further motivation is included a little later. In Chemistry, the molecular graph of a molecule, is obtained by taking vertices corresponding to the carbon atoms of the hydrocarbon system, wherein two vertices (in the graph) are adjacent if and only if there is a bond between the corresponding carbon atoms (see \cite{gg,gp}). Various energies of a molecule like HOMO energy, LUMO energy, total $\pi$-electron energy are directly related to the eigenvalues of its molecular graph. The smallest positive eigenvalue is an important parameter in Quantum Chemistry. Its magnitude is expected to be correlated with the amount of energy needed to remove an electron from the hydrocarbon molecule (see \cite{gr}). It is known that calculating the smallest positive eigenvalue is a complicated numerical task (see \cite{ranibarik} and the references cited therein). Turning our attention to the problem of determining the smallest eigenvalue, observe that when a molecular graph $G$ is non-singular, one approach to estimate bounds for the smallest positive eigenvalue of $G$, is to estimate bounds for the largest positive eigenvalue of $G^{-1}$. For the case of singular graphs, which have group inverses, it is well known that the set of all the non-zero eigenvalues of $G^{\#}$ is the set of all reciprocals of the non-zero eigenvalues of $G$. Thus, to obtain estimates for the smallest positive eigenvalue of $G$, one may consider estimates for the largest positive eigenvalue of $G^{\#}$.

Inverses of weighted graphs as defined above were studied in \cite{ghorbani,pp,sona,dybd}. A weighted graph is said to be a {\it positively weighted} graph if we assign a positive weight on each edge. A positively weighted graph with adjacency matrix $A$ is said to be {\it positively invertible} if $A^{-1}$ is diagonally similar to a non-negative matrix. Positively invertible graphs with integral edge weights have been studied in \cite{ak,sms,mn}. In a recent work, the authors of \cite{pavsi} derived a formula for the entries of the group inverse of the adjacency matrix of an undirected weighted tree and presented a graphical description.

In Section 2, we prove some of the properties of the group inverses of graphs. We prove that if $T$ is a singular weighted tree, then $T^{\#}$ is again a tree if and only if $T$ is a star tree, which in turn, holds if and only if $T^{\#}$ is graph isomorphic to $T$ that is the main result of this article. Then, we give the information of the group inverse of singular weighted paths. In Section 3, we introduce a new class of weighted trees $\mathbb{T}$ and show that this class is positively group invertible, and proved one Perron-Frobenius type result.

\begin{rem}
Henceforth for brevity, we will omit the symbol for the weight function in the exposition. We will write a weighted graph $(G,w)$ by simply $G$.
\end{rem}

\section{Properties of group inverses of weighted graphs}
At the outset, let us give a summary of the results of this section. 

Let us recall the following \cite{bapat}. Let $T$ be an (unweighted) undirected nonsingular tree (meaning that the adjacency matrix of $T$ is nonsingular). Let $A$ be the adjacency matrix of $T$. Then there exists a signature matrix $F$ such that $FA^{-1}F$ is the adjacency matrix of an (unweighted) undirected graph. Note that, this implies that the entries of $FA^{-1}F$ are either zero or one, in particular. The inverse of the nonsingular tree $T$, denoted by $T^{-1}$ is defined as the graph whose adjacency matrix is given by $FA^{-1}F$. Then $T^{-1}$ is connected \cite[Lemma 2.7]{sms} and it is bipartite \cite[Lemma 2.10]{sms}. First, we obtain a similar result for the group inverse of a singular connected weighted graph. This is presented in Proposition \ref{connected}. In Theorem \ref{fourconditions}, we proved that if $T$ is a singular weighted tree, then $T^{\#}$ is again a tree if and only if $T$ is a star tree, which in turn, holds if and only if $T^{\#}$ is graph isomorphic to $T$. Finally, we give information about the group inverse of singular paths in Proposition \ref{oddpath}.

\begin{prop}\label{connected}
Let $G$ be weighted. \\
(a) If $G$ is connected, then $G^{\#}$ is connected.\\
(b) If $G$ is bipartite, then $G^{\#}$ is bipartite.
\end{prop}
\begin{proof}
$(a)$ Let $G$ be connected and let  $A$ be the adjacency matrix of $G$. If $G^{\#}$ is not connected, then the adjacency matrix of $G^{\#}$ (assuming that it has $k$ components) will be a direct sum given by 
$A^{\#}=A_1 \oplus A_2 \oplus \cdots \oplus A_k$. This would then imply that $A=(A^{\#})^{\#}={A_1}^{\#} \oplus {A_2}^{\#} \oplus \cdots {A_k}^{\#}$, a contradiction. So, $G^{\#}$ is connected.\\
$(b)$ Since $G$ is a bipartite graph, the adjacency matrix of $G$ is given by
$$A=\begin{pmatrix}
0 & C \\
C^T & 0 \\
\end{pmatrix},$$
where $C\in \mathbb{R}^{k\times (n-k)}$ and $k\leq \frac{n}{2}$. Then by \cite[Theorem 2.2]{cov},
$$A^{\#}=\begin{pmatrix}
0 & (CC^T)^{\#}C \\
C^T(CC^T)^{\#} & \ \ 0 \\
\end{pmatrix}.$$
So, $G^{\#}$ is a bipartite graph.
\end{proof}

Next, we prove some results that will be required in subsequent discussions.

\begin{lem}\label{K1n}
Let $K_{1,n}$ be a weighted star on $\{v_1,v_2,\ldots, v_{n+1}\}$ with center $v_{n+1}$, where $w_i$ denotes the weight on the edge $v_iv_{n+1}$, for all $i=1,2,\ldots, n$. Let $p=\sum _{i=1}^{n}w_i^2$. If $A$ is the adjacency matrix of $K_{1,n}$, then $A^{\#}=\frac{1}{p}A$.
\end{lem}
\begin{proof}
The adjacency matrix $A \in \mathbb{R}^{(n+1) \times (n+1)}$ is given by $A=\begin{pmatrix}
0 & x \\
x^T & 0 \\
\end{pmatrix}$, where $x=(w_1,w_2,\ldots ,w_n)^T\in \mathbb{R}^n$.
Set $X=\frac{1}{p}A$. Since $X$ is multiple of $A$, $AX=XA$. Note that $x^Tx=p$ and so $A^3=pA$. Thus $XAX=\frac{1}{p^2}A^3=X$ and $AXA=\frac{1}{p}A^3=A$, showing that $X=A^{\#}$.
\end{proof}
Recall that an isomorphism of graphs $G$ and $H$ denoted by $G\cong H$ is a bijection between the vertex sets of $G$ and $H$, $f:V(G)\rightarrow V(H)$ such that any two vertices $u$ and $v$ of $G$ are adjacent in $G$ if and only if $f(u)$ and $f(v)$ are adjacent in $H$. For two weighted graphs $\overline{G}$ and $\overline{H}$, they are said to be isomorphic to each other if their underlying graphs are isomorphic to each other.

\begin{cor}\label{star}
Let $T$ be a weighted star. Then $T^{\#} \cong T$.
\end{cor}
\begin{proof}
Let $A$ be the adjacency matrix of $T$, then $A^{\#}$ is the adjacency matrix of $T^{\#}$. By Lemma \ref{K1n}, $A^{\#}$ is constant multiple of $A$. Therefore, $T^{\#} \cong T$.
\end{proof}

A {\it matching} in a graph is a set of edges in which no two of them have a common vertex. If this set of edges has maximum cardinality, then the matching is a {\it maximum matching}. Let $M$ be a maximum matching of a graph $G$. A path is said to be {\it alternating path} with respect to $M$ if its edges are alternatively in $M$ and $M^c$ with the first edge and the last edge being in $M$.

Let $\{v_1,v_2,\ldots , v_l\} \subseteq V(G)$. Then $G\backslash \{v_1,v_2,\ldots , v_l\}$ is the subgraph obtained from $G$ by deleting the vertices $\{v_1,v_2,\ldots , v_l\}$ together with their incident edges.

Now, we will use the following notation from \cite{pavsi}. Let $T$ be a weighted tree. For an unordered pair of vertices $v_i, v_j$ of distinct vertices of $T$ let $\mathcal{M}(v_i, v_j)$ denote the set of all maximum matchings $M$ of $T$ with the property that the $v_iv_j$-path is alternating with respect to $M$. A pair of vertices $v_i,v_j$ for which $\mathcal{M}(v_i,v_j)\neq \phi $ will be called {\it maximally matchable}.

Further, for any maximally matchable pair of vertices $v_i,v_j$ and a maximum matching $M \in \mathcal{M}(v_i, v_j)$ let $\alpha _{\overline{v_i,v_j}}(M)$ denote the product of all the weights $w(e)$, ranging over all edges $e$ of $M$ that are not contained in the unique $v_iv_j$-path $P$ in $T$. (a product over an empty set is considered to be equal to 1.) Also, for the same pair of vertices $v_i, v_j$ let $\alpha (v_i, v_j)$ be the product of all the values of $w(e)$ taken over all the edges $e$ in the path $P$, and multiplied by +1 or $-1$ depending on whether the distance between $v_i$ and $v_j$ is congruent to +1 or $-1$ mod 4; if $v_i, v_j$ are not maximally matchable set $\alpha (u_i,v_j) = 0$. For any maximally matchable pair of vertices $v_i, v_j$ of $T$ let
$$\mu _T(v_i,v_j)=\alpha (v_i,v_j) \sum _{M\in \mathcal{M} (v_i,v_j)} (\alpha _{\overline{v_i,v_j} }(M))^2.$$
It follows that $\mu _T (v_i, v_j) = 0$ if $v_i, v_j$ is not a maximally matchable pair (which includes the case $v_i = v_j$). Finally, letting $\mathcal{M}$ be the set of all maximum matchings in $T$, we let $\alpha (M)$ be the product of the weights $w(e)$ taken over all edges $e$ of $M$, for every $M \in \mathcal{M}$. Finally, set 
$$m(T)=\sum _{M\in \mathcal{M}} (\alpha (M))^2.$$

In this terminology, a formula for the group inverse of the adjacency matrix of a weighted tree, proved in \cite{pavsi}, is recalled next.

\begin{thm}\cite[Theorem 1]{pavsi}\label{pi}
Let $T$ be a weighted tree with vertex set $V$. Then, its {\it group inverse} $T^{\#}$ has two distinct vertices $v_i, v_j \in V$ joined by an edge $e$ if and only if $v_i, v_j$ is a maximally matchable pair in $T$, where the weight of $e$ is given by
$$w^{\#} (e)=w^{\#} (v_iv_j)=\frac{\mu _T(v_i,v_j)}{m(T)}.$$
\end{thm}
The understanding here is that if a pair $v_i,v_j$ is not maximally matchable (which includes the case when $v_i = v_j$), then $w^{\#} (v_iv_j)=0$.

\begin{rem}\label{consequencepi}
Let $T$ be a weighted tree and $v_i,v_j$ be two arbitrary vertices in $T$. By Theorem \ref{pi}, it is clear that there is an alternating path from $v_i$ to $v_j$ in $T$ iff $v_iv_j$ is an edge in $T^{\#}$. We will be making use of this fact, repeatedly.
\end{rem}

\begin{rem}\label{consequencepi_2}
Let $T$ be a weighted tree with vertex set $V(G)=\{v_1,v_2, . . . , v_n\}$. Then $deg(v_i)$ in $T^{\#}$ is the number of alternating paths in $T$ starting from $v_i$.
\end{rem}

\begin{lem}\label{pebelongstomm}
Let $T$ be a singular weighted tree and $e$ be a pendant edge in $T$. Then $e$ belongs to some maximum matching of $T$.
\end{lem}
\begin{proof}
Let $M$ be a maximum matching of $T$. If $e$ belongs to $M$, then we are done. So, let us suppose that $e$ does not belong to $M$. Then $e$ is adjacent to an edge of $M$; otherwise, $M \cup \{e\}$ is a matching. This, however, contradicts the fact that $M$ is a maximum matching. Let $e$ be adjacent to $v_iv_j\in M$. Since $e$ is a pendant edge, $\Tilde{M}:=(M\backslash \{v_iv_j\})\cup \{e\}$ is a maximum matching in $T$. Clearly, $e$ belongs to the maximum matching $\Tilde{M}$.
\end{proof}

\begin{lem}\label{nottree}
Let $T$ be a weighted tree other than the star. If $T$ has at least two adjacent pendant edges, then $T^{\#}$ is not a tree.
\end{lem}
\begin{proof}
Observe that since $T$ has at least two adjacent pendant edges, it must be singular. Let $v_iv_p$ and $v_iv_q$ be adjacent pendant edges of $T$. Since $T$ is not a star, it has a path of length at least three. So, the number of edges in a maximum matching will be at least two. Let $M$ be a maximum matching in $T$ that contains the pendant edge $v_iv_p$ (This is possible because of Lemma \ref{pebelongstomm}).

{\bf Case (i)}: Suppose that there exists an alternating path $P$ in $T$, of length at least three with respect to $M$ that has $v_iv_p$ as the initial edge. Note that the initial vertex of $P$ is $v_p$; let the terminal vertex be $v_k$. Since $v_iv_q$ is a pendant edge and adjacent to $v_iv_p$, $\Tilde{M}:=(M\backslash \{v_iv_p\})\cup \{v_iv_q\}$ is also a maximum matching in $T$. The path $Q$ with initial vertex $v_q$ and terminal vertex $v_k$ is an alternating path with respect to $\Tilde{M}$ in $T$. Observe that, $P$ and $Q$ are identical except for the initial edge. Now, by using Remark \ref{consequencepi}, the vertex sequence $(v_i,v_p,v_k,v_q,v_i)$ is a cycle in $T^{\#}$, proving that $T^{\#}$ is not a tree.

{\bf Case (ii)}: Suppose that there does not exist any alternating path of length at least three with respect to $M$ that has $v_iv_p$ as an initial edge. Then there exists a path $v_pv_iv_lv_mv_j$ of length four in $T$ such that edges $v_pv_i,v_mv_j$ belong to $M$. Observe that $\overline{M}=(M\backslash \{v_mv_j\})\cup \{v_lv_m\}$ will also be a maximum matching in $T$. Also, $\widehat {M}=(\overline{M}\backslash \{v_iv_p\})\cup \{v_iv_q\}$ is a maximum matching in $T$. So, the path $v_pv_iv_lv_m$ is an alternating path with respect to $\overline{M}$ and the path $v_qv_iv_lv_m$ is an alternating path with respect to $\widehat{M}$ in $T$. Again, by using Remark \ref{consequencepi}, the vertex sequence $(v_i,v_p,v_m,v_q,v_i)$ will form a cycle in $T^{\#}$, and so $T^{\#}$ is not a tree.
\end{proof}
% \TR{The next result and its corollary (newly included) do not depend on Theorem \ref{pi}, right? So, can we move it to the previous section? Of course, Lemma \ref{pebelongstomm} also should be moved, along with.}\\

% \TR{Paraphrasing Theorem \ref{nottree}, we obtain the following:}

% \TR{\begin{thm}\label{conversenottree}
% Let $T$ be a weighted tree, which is not a star. If $T^{\#}$ is a tree, then no two pendant edges are adjacent in $T$.
% \end{thm}}

\begin{rem}
Lemma \ref{nottree} may be paraphrased as follows. Let $T$ be not a star. If $T$ and $T^{\#}$ are weighted trees, then no two pendant edges are adjacent in $T$. Example \ref{wtpathgpinv}, given later, shows that the converse of Lemma \ref{nottree} is not true.
\end{rem}

\begin{lem}\label{contain4cycle}
Let $T$ be a singular weighted tree. Let $T$ have a length three alternating path $P$ \rm{(}relative to some maximum matching\rm{)} such that the middle non-matching edge of $P$ belongs to another maximum matching. Then $T^{\#}$ is not a tree.
\end{lem}
\begin{proof}
Let $P$ be the path $v_iv_pv_qv_j$ which is alternating with respect to a maximum matching $M$. Let the middle edge $v_pv_q$ belong to another maximum matching $\Tilde{M}$. Since $v_iv_p$ and $v_qv_j$ are both length one alternating paths with respect to $M$, they are also edges in $T^{\#}$, by Remark \ref{consequencepi}. Now, the vertex sequence $(v_i,v_p,v_q,v_j,v_i)$ is a 4-cycle in $T^{\#}$. Thus, $T^{\#}$ is not a tree.
\end{proof}

Recall that a {\it corona tree} is a tree obtained by attaching a new pendant vertex to each vertex of a given tree, and corona trees are always non-singular.

For a non-singular undirected unweighted tree $T$, it is shown in \cite[Theorem 2.11]{sms} that $T^{-1}$ is again a tree if and only if $T$ is a corona tree. In the next Lemma, we obtain a similar characterization for singular weighted trees.
\begin{lem}\label{lasttree}
Let $T$ be a singular weighted tree. Then, $T^{\#}$ is again a tree if and only if $T$ is a star tree.
\end{lem}
\begin{proof}
Suppose $T$ is a star, then by Lemma \ref{star}, $T^{\#}$ is isomorphic to $T$. For the converse part, we prove it contrapositively, i.e. we show that if $T$ is not a star then $T^{\#}$ is not a tree. Since $T$ is a singular weighted tree, it is also not a corona tree. Then there exists a non-pendant vertex $v_k$ such that either it is adjacent to more than one pendant vertex or not adjacent to any pendant vertex. For the first case, by Lemma \ref{nottree}, $T^{\#}$ is not a tree. Now, we discuss the later case. Suppose $v_k$ is adjacent to $m$ non-pendant vertices $v_{kj}$, $1\leq j\leq m$. Let $M$ be a maximum matching of $T$.

{\bf Case (i):} None of the edges $v_kv_{kj}$ belong to $M$. Then there exists an edge $v_{kj}v_{\Tilde{kj}}$, $1\leq j\leq m$ other than $v_{kj}v_k$ such that $\{v_{kj}v_{\Tilde{kj}} : 1\leq j\leq m\}\subseteq M$ otherwise, it will contradict the fact that $M$ is a maximum matching. Now, $M_1=(M\backslash \{v_{k1}v_{\Tilde{k1}}\})\cup \{v_kv_{k1}\}$ and $M_2=(M\backslash \{v_{k2}v_{\Tilde{k2}}\})\cup \{v_kv_{k2}\}$ are also maximum matchings of $T$. Hence, $v_{k1}v_kv_{k2}v_{\Tilde{k2}}$ is an alternating path with respect to $M_1$ and $v_kv_{k2}$ is a matching edge in $M_2$. Then, By Lemma \ref{contain4cycle}, $T^{\#}$ is not a tree.

{\bf Case (ii):} For some $i$, $v_kv_{ki}$ belongs to $M$. Then, there exists (if not, we can construct such kind of maximum matching always) a maximum matching $\widehat {M}$ such that $\{v_kv_{ki}\}\cup \{v_{kj}v_{\overline{kj}}: 1\leq j\leq m,j\neq i\})\subseteq \widehat {M}$, where $v_{kj}v_{\overline{kj}}$ is some edge adjacent to $v_{kj}$ other than $v_{kj}v_k$, $1\leq j\leq m$.

\underline{Subcase (i)}: If $v_{ki}v_{\overline{ki}}$ is a pendant edge. Then, $\overline{M}=(\widehat{M}\backslash \{v_kv_{ki},v_{kp}v_{\overline{kp}}\})\cup \{v_{ki}v_{\overline{ki}},v_kv_{kp}\}$ is also a maximum matching in $T$ where $p\neq i$. Again, by Lemma \ref{contain4cycle}, $T^{\#}$ is not a tree.

\underline{Subcase (ii)}: Suppose $v_{ki}v_{\overline{ki}}$ is not a pendant edge. Then, there exists a maximum matching $\Tilde{M}$ such that $\{v_kv_{ki},v_{\overline{ki}}v_{\widehat{ki}}\}\cup \{v_{kj}v_{\overline{kj}} : 1\leq j\leq m,j\neq i\})\subseteq \Tilde {M}$, where $v_{\overline{ki}}v_{\widehat{ki}}$ is some edge adjacent to $v_{\overline{ki}}$ other than $v_{ki}v_{\overline{ki}}$. So, the path $v_{\overline{kq}}v_{kq}v_kv_{ki}v_{\overline{ki}}v_{\widehat{ki}}$ is a length five alternating path with respect to $\Tilde{M}$ for some $q\neq i$, which leads to forming a cycle of length four in $T^{\#}$.

Hence, this completes the proof.
\end{proof}

In \cite[Theorem 2.11]{sms}, the authors proved that for a non-singular tree $T$, $T^{-1}$ is again a tree if and only if $T$ is a corona tree, which in turn, holds if and only if $T^{-1}$ is isomorphic to $T$. It is natural to ask if a similar looking characterization applies when the tree is singular. In the next result, we identify a class of singular trees for which an analogue holds.

\begin{thm}\label{fourconditions}
For a singular weighted tree $T$ with $n$ vertices, the following statements are equivalent:\\
(i) The number of alternating paths in $T$ is $n-1$.\\
(ii) $T^{\#}$ is a tree. \\
(iii) $T$ is the star graph. \\
(iv) $T^{\#}$ is isomorphic to $T$.
\end{thm}
\begin{proof}
$(i)\Rightarrow (ii)$: Since $T$ has $n-1$ alternating paths, by Remark \ref{consequencepi}, $T^{\#}$ has $n-1$ edges. Since $T^{\#}$ is connected and has $n-1$ edges, $T^{\#}$ is a tree.\\
$(ii)\Rightarrow (iii)$: It is clear from Lemma \ref{lasttree}.\\
$(iii)\Rightarrow (iv)$: Follows from Corollary \ref{star}.\\
$(iv)\Rightarrow (i)$: Let $T$ be not the star graph. By Lemma \ref{lasttree}, $T^{\#}$ is not a tree. So, $T$ is not isomorphic to $T^{\#}$. However, since $T^{\#}$ is assumed to be isomorphic to $T$, we conclude that $T$ is a star graph. Since each singleton edge set will form a maximum matching in a star, the number of alternating paths in a star is the number of edges, viz., $n-1$.
\end{proof}
\begin{cor}\label{zerononzero}
Let $T$ be a singular weighted tree and $A$ be the adjacency matrix of $T$. Then $A$ and $A^{\#}$ have the same zero non-zero pattern if and only if $T$ is a weighted star.
\end{cor}
\begin{proof}
Note that $A$ and $A^{\#}$ have the same zero non-zero pattern means $T$ and $T^{\#}$ are isomorphic to each other. Then, the proof follows from Lemma \ref{K1n} and Theorem \ref{fourconditions}.
\end{proof}
We say that a weighted undirected graph $H$ is a subgraph of a weighted undirected graph $G$ if the underlying graph of $H$ is a subgraph of the underlying graph of $G$.

In the next result, we provide some information on the group inverse of odd paths. We denote a path by $P_n$ and its edge set by $\{v_iv_{i+1}: 1 \leq i\leq n-1\}$. Then $\widehat {M}=\{v_1v_2,v_3v_4, \ldots , v_{2n-1}v_{2n} \}$ and $\overline{M}=\{v_2v_3, v_4v_5, \ldots , v_{2n}v_{2n+1} \}$ are two maximum matchings in $P_{2n+1}$.

We conclude this section with a result on the group inverses of odd paths. 

\begin{prop}\label{oddpath}
Consider $P_{2n+1}$ as a weighted graph. Then, we have the following:\\
(i) $P_{2n+1}$ is a spanning subtree of $P_{2n+1}^{\#}$.\\
(ii) $v_iv_j$ is an edge in $P_{2n+1}^{\#}$ if and only if $i+j$ is odd.\\
(iii) For $n\geq 2$, the graph $P_{2n+1}^{\#}$ does not have any pendant vertices.
\end{prop}
\begin{proof}
$(i)$ Since the maximum matching edges are also alternating paths, it follows by Remark \ref{consequencepi} that $v_iv_{i+1}$ is an edge in $P_{2n+1}^{\#}$, for all $i=1, 2,\ldots, 2n$. Therefore, $P_{2n+1}$ is a spanning subtree of $P_{2n+1}^{\#}$.

$(ii)$ Suppose that $i+j$ is odd. Then $i$ and $j$ have different parities. Every edge of a path in $P_{2n+1}$ belongs to either $\widehat {M}$ or $\overline{M}$. A path starting from $v_i$, with $i$ odd and ending with $v_j,$ where $j$ is even, is alternating with respect to the maximum matching $\widehat {M}$, whereas a path starting from $v_i$, when $i$ is even and ending with $v_j$, with $j$ being odd, is alternating with respect to $\overline{M}$. So, in both the cases, $v_iv_j$ is an edge in $P_{2n+1}^{\#}$, which follows again, by Remark \ref{consequencepi}.

We will prove the converse part by contrapositive equivalence. Let $i+j$ be even. Then $i$ and $j$ have the same parity. Then the length of any path starting from $v_i$ and ending in $v_j$ is even. So, the $v_iv_j$-path cannot be an alternating path. Thus, $v_iv_j$ cannot be an edge in $P_{2n+1}^{\#}$, by Remark \ref{consequencepi}. Therefore, if $v_iv_j$ is an edge in $P_{2n+1}^{\#},$ then $i+j$ is odd.

$(iii)$ Since the edge $v_1v_2$ and  the $v_1v_4$-path are two alternating paths starting from $v_1$ with respect to maximum matching $\widehat {M}$, the degree of the vertex $v_1$ in $P_{2n+1}^{\#}$ is at least two. Similarly, the $v_{2n}v_{2n+1}$-alternating path and the $v_{2n-2}v_{2n+1}$-alternating path with respect to maximum matching $\overline{M}$ implies the degree of the vertex $v_{2n+1}$ in $P_{2n+1}^{\#}$ is at least two.

For any vertex $v_i,i\neq 1,2n+1$, the $v_{i-1}v_i$ path and the $v_iv_{i+1}$ path are two alternating paths in $P_{2n+1}$. This implies the degree of $v_i$ in $P_{2n+1}^{\#}$ is also at least two. So, the degree of any vertex in $P_{2n+1}^{\#}$ is at least two. Therefore, there is no pendant vertex in $P_{2n+1}^{\#}$.
\end{proof}
\newpage
\begin{Exm}\label{wtpathgpinv}
Consider the weighted path $P_5$ and its group inverse $P_5^{\#}$.
\begin{figure}[H]
\begin{tikzpicture}[scale=1]
\draw  (0.,0.)-- (1.52,0.);
\draw  (1.52,0.)-- (3.,0.);
\draw  (3.,0.)-- (4.5,0.);
\draw  (4.5,0.)-- (6.,0.);
\begin{scriptsize}
\fill (0.,0.) circle (2.5pt);
\draw (0.14,0.33) node {$v_1$};
\fill (1.52,0.) circle (2.5pt);
\draw (1.66,0.37) node {$v_2$};
\draw (0.8,-0.35) node {$1$};
\fill (3.,0.) circle (2.5pt);
\draw (3.14,0.37) node {$v_3$};
\draw (2.3,-0.35) node {$1$};
\fill (4.5,0.) circle (2.5pt);
\draw (4.64,0.37) node {$v_4$};
\draw (3.8,-0.35) node {$1$};
\fill (6.,0.) circle (2.5pt);
\draw (6.14,0.37) node {$v_5$};
\draw (5.3,-0.35) node {$1$};
\end{scriptsize}
\end{tikzpicture}
\hspace{2.5cm}
\begin{tikzpicture}[scale=1]
\draw  (2.,1.)-- (2.,-1.);
\draw  (2.,1.)-- (4.,1.);
\draw  (4.,-1.)-- (4.,1.);
\draw  (2.,-1.)-- (4.,-1.);
\draw  (1.,0.)-- (4.,1.);
\draw  (1.,0.)-- (2.,-1.);
\begin{scriptsize}
\fill (2.,1.) circle (2.5pt);
\draw (2.14,1.37) node {$v_1$};
\fill (2.,-1.) circle (2.5pt);
\draw (2.14,-1.3) node {$v_4$};
\draw (1.77,-0.37) node {$-\frac{1}{3}$};
\fill (4.,1.) circle (2.5pt);
\draw (4.14,1.37) node {$v_2$};
\draw (3.04,1.25) node {$\frac{2}{3}$};
\fill (4.,-1.) circle (2.5pt);
\draw (4.14,-1.3) node {$v_3$};
\draw (4.27,0.15) node {$\frac{1}{3}$};
\draw (3.04,-1.38) node {$\frac{1}{3}$};
\fill (1.,0.) circle (2.5pt);
\draw (1.14,0.37) node {$v_5$};
\draw (2.66,0.29) node {$-\frac{1}{3}$};
\draw (1.12,-0.55) node {$\frac{2}{3}$};
\end{scriptsize}
\end{tikzpicture}
\caption{$P_5$ (Left) and $P_5^{\#}$ (Right)}
 \label{path}
\end{figure}
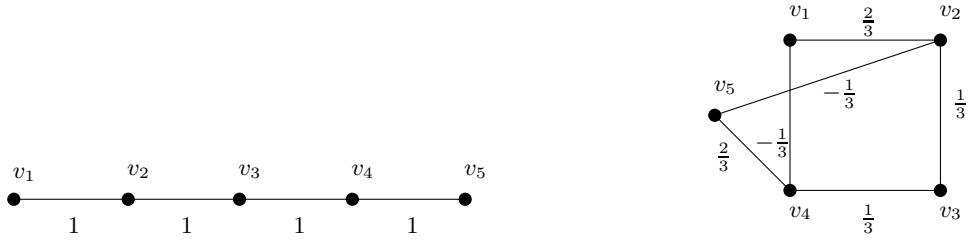
We record the following observations:\\
(i) The $v_1v_2v_3v_4v_5$ path, which is $P_5$ here, is a spanning subtree of $P_5^{\#}$.\\
(ii) $v_iv_j$ is an edge in $P_5^{\#}$ if and only if $i+j$ is odd.\\
(iii) $P_5^{\#}$ has no pendant vertices.\\
(iv) Evidently, $P_5$ has three maximum matchings \rm{(}each having two edges\rm{)}. The hypothesis of Lemma \ref{contain4cycle} holds. Observe that $P_{2n+1}^{\#}$ contains a 4-cycle.
\end{Exm}

\section{A new class of weighted trees}
In this section, we introduce a new class of weighted trees and show that the group inverse of the adjacency matrix of a positively weighted tree in this class is signature similar to a non-negative matrix. It is also shown that the smallest positive eigenvalue of the adjacency matrix of any member in the given class, with positive weight on each edge, is simple.

Now, we define a new class of weighted trees. The class $\mathbb{T}$ is defined by the requirement that each $T \in \mathbb{T}$ is a simple, undirected, weighted tree having properties that\\
(i) $T$ has at least one non-pendant vertex.\\
(ii) Each non-pendant vertex is adjacent to at least one pendant vertex in $T$.
\begin{Exm}
Consider the two weighted trees $T_1$ and $T_2$.
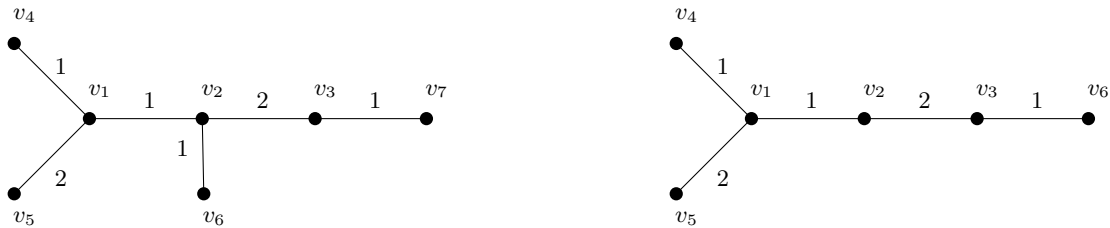
\begin{figure}[H]\label{newtree}
\begin{tikzpicture}[scale=1]
\draw  (1.,1.)-- (2.,0.);
\draw  (2.,0.)-- (1.,-1.);
\draw  (2.,0.)-- (3.5,0.);
\draw  (3.5,0.)-- (3.52,-1.);
\draw  (3.5,0.)-- (5.,0.);
\draw  (5.,0.)-- (6.48,0.);
\begin{scriptsize}
\fill (1.,1.) circle (2.5pt);
\draw (1.14,1.37) node {$v_4$};
\fill (2.,0.) circle (2.5pt);
\draw (2.14,0.37) node {$v_1$};
\draw (1.62,0.7) node {$1$};
\fill (1.,-1.) circle (2.5pt);
\draw (1.14,-1.33) node {$v_5$};
\draw (1.62,-0.8) node {$2$};
\fill (3.5,0.) circle (2.5pt);
\draw (3.64,0.37) node {$v_2$};
\draw (2.8,0.25) node {$1$};
\fill (3.52,-1.) circle (2.5pt);
\draw (3.66,-1.33) node {$v_6$};
\draw (3.24,-0.4) node {$1$};
\fill (5.,0.) circle (2.5pt);
\draw (5.14,0.37) node {$v_3$};
\draw (4.3,0.25) node {$2$};
\fill (6.48,0.) circle (2.5pt);
\draw (6.62,0.37) node {$v_7$};
\draw (5.8,0.25) node {$1$};
\end{scriptsize}
\end{tikzpicture}
\hspace{2.5cm}
\begin{tikzpicture}[scale=1]
\draw  (1.,1.)-- (2.,0.);
\draw  (2.,0.)-- (1.,-1.);
\draw  (2.,0.)-- (3.5,0.);
\draw  (3.5,0.)-- (5.,0.);
\draw  (5.,0.)-- (6.48,0.);
\begin{scriptsize}
\fill (1.,1.) circle (2.5pt);
\draw (1.14,1.37) node {$v_4$};
\fill (2.,0.) circle (2.5pt);
\draw (2.14,0.37) node {$v_1$};
\draw (1.62,0.7) node {$1$};
\fill (1.,-1.) circle (2.5pt);
\draw (1.14,-1.33) node {$v_5$};
\draw (1.62,-0.8) node {$2$};
\fill (3.5,0.) circle (2.5pt);
\draw (3.64,0.37) node {$v_2$};
\draw (2.8,0.25) node {$1$};
\fill (5.,0.) circle (2.5pt);
\draw (5.14,0.37) node {$v_3$};
\draw (4.3,0.25) node {$2$};
\fill (6.48,0.) circle (2.5pt);
\draw (6.62,0.37) node {$v_6$};
\draw (5.8,0.25) node {$1$};
\end{scriptsize}
\end{tikzpicture}
\caption{$T_1$ (Left) and $T_2$ (Right)}
\end{figure}
Observe that $T_1 \in \mathbb{T}$ and $T_2 \notin \mathbb{T}$. $T_1$ has 3 non-pendant vertices, $v_1,v_2$ and $v_3$ and they are each adjacent to at least one pendant vertex,  however, the non-pendant vertex $v_2$ in $T_2$ is not adjacent to any pendant vertex.
\end{Exm}

It is pertinent to point to the fact that the underlying graphs (belonging to $\mathbb{T}$) are a particular case of cluster networks obtained by taking an arbitrary tree as the base and stars as satellites, \cite{ca,lyz}. In \cite{lyz}, the Kirchhoff index formulae for
a cluster of two graphs, are presented, in terms of the pieces. The Kirchhoff index formulae and the effective resistances of generalized cluster networks are obtained, in terms of the pieces, in \cite{ca}. Cluster networks are highly relevant in chemistry applications since all composite molecules consisting of some amalgamation over a central submolecule can be understood as a generalized cluster network. For instance, they can be used to understand some issues
on metal-metal interaction in some molecules (see \cite{toma}), since a cluster network structure can be easily found.

In \cite{godsil}, Godsil has shown that if $A$ is the adjacency matrix of an undirected unweighted bipartite graph $G$ with a unique perfect matching $M$ and $G\backslash M$ is again a bipartite graph, then $A^{-1}$ is {\it diagonally similar} to a non-negative matrix. In particular, for an invertible (unweighted) tree, the inverse (of the adjacency matrix) is diagonally similar to a $\{0,1\}$-matrix. Also, he raised the problem of characterizing all bipartite graphs with a unique perfect matching such that the inverse of the adjacency matrix of each such graph is diagonally similar to a non-negative matrix. Later on in \cite{pp}, Panda and Pati generalized Godsil's result for the case of weighted graphs. By assigning a weight of $1$ to each matching edge and assuming arbitrary positive weights on non-matching edges, they showed that $A^{-1}$ is signature similar to a non-negative matrix. The problem raised by Godsil, in its full generality, was recently solved in \cite[Theorem 1.3]{yd}. A similar question for the class of singular bipartite graphs has an affirmative answer for the new class of trees introduced here with a positive weight on each edge.

\begin{rem}
Note that weighted corona trees belong to $\mathbb{T}$. Consider a tree $T\in \mathbb{T}$ which is not a corona tree. Then there are at least two pendant vertices having a common neighbour in $T$. This implies that at least two rows in the adjacency matrix of $T$ are multiples of each other. Hence $T$ is singular. 
\end{rem}

\begin{thm}\label{pendantedge}
Let $T \in \mathbb{T}$. Then no non-pendant edge can belong to a maximum matching of $T$.
\end{thm}
\begin{proof}
If $T$ is a corona tree then it has a perfect matching and each matching edge is a pendant edge. Now, we consider the case where $T$ is not a corona tree. In that case, there is at least one non-pendant vertex which is adjacent to at least two pendant vertices in $T$. We prove the assertion by induction on the number of vertices in $T$.

The smallest tree in $\mathbb{T}$ is $K_{1,2}$ and every maximum matching has only pendant edges. Let $T \in \mathbb{T}$ have  $n$ vertices. Let the statement be true for any tree that belongs to $\mathbb{T}$ and whose the number of vertices is less than $n$. Let $v_i$ be a non-pendant vertex adjacent to $r$ pendant vertices $\{v_{i_1},v_{i_2},\ldots ,v_{i_r}\}$ in $T$. Let $e$ be an arbitrary non-pendant edge contained in a maximum matching $M$ in $T$. We show that this leads to a contradiction.

{\bf Case (i):} $e$ is incident to $v_i$. Then none of the edges $v_iv_{i_p}$, $1 \leq p\leq r$,  belongs to $M$. So, $M$ will also be a maximum matching of the tree $T\backslash \{v_{i_1}\}\in \mathbb{T}$. This, however contradicts the induction hypothesis that a non-pendant edge belongs to a maximum matching in the tree $T\backslash \{v_{i_1}\}$ (with $n-1$ vertices).

{\bf Case (ii):} $e$ is not incident with vertex $v_i$. Then one of $v_iv_{i_p}, 1 \leq p \leq r$ belongs to $M$; otherwise $M$ will not be a maximum matching. Let $v_iv_{i_p}$ belong to $M$ for some $p, 1 \leq p\leq r$. Then $M$ will also be a maximum matching for the tree $T\backslash \{v_{i_q}\}$ for some $q\neq p$, which again contradicts the induction hypothesis.

This completes the proof.
\end{proof}

We have the following immediate consequence, which will prove to be quite useful in further considerations.

\begin{cor}\label{length}
Let $T\in \mathbb{T}$. Then the length of any alternating path is at most three.
\end{cor}
\begin{proof}
Suppose $T$ has an alternating path $P$ of length at least five. Then $P$ should have at least one non-pendant edge, belonging to a maximum matching. This is a contradiction.
\end{proof}

\begin{rem}
Let $T\in \mathbb{T}$ have $k$ non-pendant vertices $\{v_1,v_2,\ldots ,v_k\}$ where each $v_i$ is adjacent to $t_i$ pendant vertices. Then any maximum matching of $T$ always contains a set of $k$ non-adjacent pendant edges. So, the number of edges in a maximum matching will be $k$ and the number of maximum matchings will be $t_1t_2\ldots t_k$ in $T$.
\end{rem}

\begin{rem}\label{lengthproperty}
Let $T\in \mathbb{T}$. Then both the initial and the terminal vertices of a length three alternating path in $T$ should be pendant vertices; a length one alternating path is nothing but a pendant edge.
\end{rem}

\begin{cor}\label{dc}
Let $T \in \mathbb{T}$. We then have the following:\\
(a) Let $v_1,v_2, . . . , v_k$ denote the non-pendant vertices where each $v_i$ is adjacent to $t_i$ pendant vertices. Then in $T^{\#}$, one has the formula $deg(v_i)=t_i$.\\
(b) Let $T$ be a tree other than the star and the corona tree. Then $T^{\#}$ has at least one 4-cycle.
\end{cor}
\begin{proof}
$(a)$ By Theorem \ref{pendantedge}, non-pendant edges do not belong to any maximum matching in $T$. This means that the $t_i$ pendant edges are the only alternating paths starting from $v_i$ in $T$. By Remark \ref{consequencepi_2}, the number of edges incident to $v_i$ in $T^{\#}$ is $t_i$. Therefore, $deg(v_i)$ in $T^{\#}$ is $t_i$.

$(b)$ Since $T$ is not a corona tree and $T \in \mathbb{T}$, there exist two adjacent pendant edges in $T$. Now, from the proof of Lemma \ref{nottree}, it is clear that $T^{\#}$ has at least one 4-cycle.
\end{proof}

In what follows, we extend the result of \cite{godsil} for invertible trees, mentioned earlier, to the class $\mathbb{T}$. Here, one thing we need to mention is that if we consider the base tree of $T$ is nonsingular and attach the same number of pendant vertex to each non-pendant vertex in the following Theorem, then the signability of $A^{\#}$ can be derived from \cite[Proposition 6]{pav}. The authors of \cite{pav} proved their result linear algebraically while we give a simple graph-theoretic proof.

\begin{thm}\label{signaturesimilar}
Let $T\in \mathbb{T}$ be assigned positive weights on each edge. Let $A$ be the adjacency matrix of $T$. Then there exists a signature matrix $S$ such that $SA^{\#}S$ is a non-negative matrix.
\end{thm}
\begin{proof}
Without loss of generality, let $v_1$ be a pendant vertex in $T$. For $i=1,2,\ldots ,n$, let $P(v_1,v_i)$ be the path from $v_1$ to $v_i$ and $n_i$ be the number of edges on the path $P(v_1,v_i)$ which does not belong to any maximum matching in $T$. Set $n_1=0,~s_i=(-1)^{n_i}$, $i=1,2,\ldots ,n$ and let $S=$ diag$(s_1,s_2,\ldots ,s_n)$. Let $A^{\#}=(\beta _{ij})$, then $ij$-th element of $SA^{\#}S$ is $s_is_j\beta _{ij}$, which equals to zero if and only if $\beta _{ij}=0$. Suppose $\beta _{ij}\neq 0$, then by Theorem \ref{pi}, the path between $v_i$ and $v_j$ is alternating with respect to some maximum matching. Since the edge weights are positive, the sign of $\beta _{ij}$ will depend on the sign of $(-1)^{\frac{d(v_i,v_j)-1}{2}}$. By Corollary \ref{length}, $d(v_i,v_j)$ is either one or three. If $d(v_i,v_j)=1$, then $v_iv_j$ is a pendant edge and so, $n_i = n_j$ . Now, $s_is_j=(-1)^{2n_i}=1$ and $(-1)^{\frac{d(v_i,v_j)-1}{2}}=1$. Thus, the
$ij$-th element of $SA^{\#}S$ is positive. On the other hand, when $d(v_i,v_j)=3$, either $n_j = n_i + 1$ or $n_j = n_i-1$. In both these situations $s_is_j=-1$ and $(-1)^{\frac{d(v_i,v_j)-1}{2}}=-1$. Thus, the $ij$-th element of $SA^{\#}S$ is positive, in this case, too.
\end{proof}
\begin{rem}
We show by means of an example that the result above does not hold for singular weighted trees which do not belong to $\mathbb{T}$. 
\begin{figure}[H]
\begin{tikzpicture}[scale=1]
\draw  (1.,0.)-- (3.,0.);
\draw  (3.,0.)-- (5.,0.);
\draw  (5.,0.)-- (7.,0.);
\draw  (7.,0.)-- (9.,1.);
\draw  (7.,0.)-- (9.,-1.);
\begin{scriptsize}
\fill (1.,0.) circle (2.5pt);
\draw (1.,0.37) node {$v_1$};
\fill (3.,0.) circle (2.5pt);
\draw (3.,0.37) node {$v_2$};
\fill (5.,0.) circle (2.5pt);
\draw (5.,0.37) node {$v_3$};
\fill (7.,0.) circle (2.5pt);
\draw (7.,0.37) node {$v_4$};
\fill (9.,1.) circle (2.5pt);
\draw (9.35,1.05) node {$v_5$};
\fill (9.,-1.) circle (2.5pt);
\draw (9.35,-1.) node {$v_6$};
\draw (2.06,0.25) node {$1$};
\draw (4.06,0.25) node {$1$};
\draw (6.06,0.25) node {$1$};
\draw (8.1,0.95) node {$1$};
\draw (8.2,-0.9) node {$1$};
\end{scriptsize}
\end{tikzpicture}
\caption{$(T,w)$}
\label{notgodsil}
\end{figure}
The adjacency matrix of $T$ is 
$$A=\begin{pmatrix}
0 & 1 & 0 & 0 & 0 & 0 \\
1 & 0 & 1 & 0 & 0 & 0 \\
0 & 1 & 0 & 1 & 0 & 0 \\
0 & 0 & 1 & 0 & 1 & 1 \\
0 & 0 & 0 & 1 & 0 & 0 \\
0 & 0 & 0 & 1 & 0 & 0 \\
\end{pmatrix}$$ and 
$$A^{\#}=\begin{pmatrix}
~0 & ~\frac{3}{5} & 0 & -\frac{1}{5} & ~0 & ~0 \\
~\frac{3}{5} & ~0 & \frac{2}{5} & ~0 & -\frac{1}{5} & -\frac{1}{5} \\
~0 & ~\frac{2}{5} & 0 & ~\frac{1}{5} & ~0 & ~0 \\
-\frac{1}{5} & ~0 & \frac{1}{5} & ~0 & ~\frac{2}{5} & ~\frac{2}{5} \\
~0 & -\frac{1}{5} & 0 & ~\frac{2}{5} & ~0 & ~0 \\
~0 & -\frac{1}{5} & 0 & ~\frac{2}{5} & ~0 & ~0 \\
\end{pmatrix}.$$ 
Suppose that there exist a signature matrix $S=(s_{ij})$ such that $SA^{\#}S$ is non-negative. Let $A^{\#}=(\beta _{ij})$ and $B=SA^{\#}S=(b_{ij})$. Then $b_{ij}=s_is_j\beta _{ij}$, so that $b_{12}=\frac{3}{5}s_1s_2$, $b_{14}=-\frac{1}{5}s_1s_4$, $b_{23}=\frac{2}{5}s_2s_3$ and $b_{34}=\frac{2}{5}s_3s_4$. Since $B$ is a non-negative matrix, the numbers $s_1s_2$, $s_1s_4$, $s_2s_3$ and $s_3s_4$ are positive, negative, positive and positive, respectively. Since $s_1s_2$ is positive and $s_1s_4$ is negative, we infer that $\frac{s_2}{s_4}$ is negative. Further, since $s_2s_3$ and $s_3s_4$ positive, one obtains that $\frac{s_2}{s_4}$ is positive, a contradiction.
\end{rem}

\begin{rem}\label{irreducible}
Let us recall that $A$ is said to be an irreducible matrix if digraph $D(A)$ corresponding to matrix $A$ is strongly connected. Let $A$ be an $n\times n$ real matrix such that $A^{\#}$ exists. Then by \cite[Lemma 2.4]{kls}, $A$ is irreducible if, and only if, $A^{\#}$ is irreducible .
\end{rem}

In \cite[Corollary 2.8]{sms}, it was shown that for a non-singular unweighted tree $T$, the smallest positive eigenvalue of $T$ is simple and there exists an eigenvector corresponding to the smallest eigenvalue such that all its entries are non-zero. We prove a similar result for the new class of trees $\mathbb{T}$, with an arbitrary assignment of positive weights on each edge.

The spectral radius, denoted by $\rho (G)$, of $G$ is the spectral radius of its adjacency matrix $A$.

\begin{cor}
Let $T\in \mathbb{T}$ with a positive weight on each edge with the adjacency matrix $A$. Then the smallest positive eigenvalue of $T$, $\tau (T)$ is simple. Furthermore, there exists a corresponding eigenvector each of whose entries is nonzero. 
\end{cor}
\begin{proof}
By Theorem \ref{signaturesimilar}, there exists a signature matrix $S$ such that $SA^{\#}S$ is a non-negative matrix. Since $A$ is irreducible, by Remark \ref{irreducible}, $A^{\#}$ is irreducible and so $SA^{\#}S$ is irreducible. Now, by the Perron Frobenius theorem, $\rho:=\rho(SA^{\#}S)$ is simple and since $SA^{\#}S$ is similar to $A^{\#}$, $\rho(T^{\#})$ is simple. So, $\tau(T)=\frac{1}{\rho(T^{\#})}$ is simple.

Let $x > 0$ be the Perron vector corresponding to $\rho (SA^{\#}S)$ so that $SA^{\#}S x=\rho (SA^{\#}S) x$. Since $\rho (SA^{\#}S)=\rho(A^{\#})$ and $S=S^{-1}$, $A^{\#}(Sx)=\rho (T^{\#})(Sx)$. So, $Sx\in R(A^{\#})=R(A)$ and so, $Sx=AA^{\#}(Sx)=\rho (T^{\#})A(Sx)$. This means that $\tau (T)(Sx)=\frac{1}{\rho (T^{\#})}(Sx)=A(Sx),$ showing that $Sx$ is an eigenvector corresponding to $\tau (T)$. Since $x>0$ and $S$ is a signature matrix, all the entries of $Sx$ are non-zero.
\end{proof}
We conclude this section by giving information on trees in $\mathbb{T}$, that are caterpillars. A caterpillar tree is a tree for which removing the leaves and incident edges produces a path graph.

\begin{prop}
Let $T \in \mathbb{T}$ be a weighted caterpillar other than a star with $n$ pendant vertices and $k$ non-pendant vertices $v_1,v_2,\ldots ,v_k$ such that $v_i\sim v_{i+1}$, for $i=1,2,\ldots , k-1$. Let $t_i$ be the number of pendant vertices adjacent to $v_i$ for each $i$. Then the number of edges in $T^{\#}$ is $n+t_1t_2+t_2t_3+\cdots +t_{k-1}t_k.$
\end{prop}
\begin{proof}
We have assumed that $v_i\sim v_{i+1}$, for $i=1,2,\ldots , (k-1)$. By Remark \ref{consequencepi}, the number of alternating paths in $T$ is the number of edges in $T^{\#}$. Since the maximum matching edges are the only pendant edges, the number of length one alternating path in $T$ is $n$.

For a length three alternating path, the initial edge should be incident to $v_i$ and the terminal edge should be incident to $v_{i+1}$ for some $i\in \{1,2,\ldots , k-1\}$. For any $i$, the total number of 
alternating paths starting from a pendant vertex adjacent to $v_i$ and ending at a pendant vertex adjacent to $v_{i+1}$ in $T$ is $t_it_{i+1}$. So, the total number of length three alternating paths in $T$ is $t_1t_2+t_2t_3+ \cdots +t_{k-1}t_k$. Since the length of any alternating path is at most three in $T$, the total number of edges in $T^{\#}$ is $n+t_1t_2+t_2t_3+ \cdots + t_{k-1}t_k$.
\end{proof}

\vspace{1cm}

\section{Acknowledgement}

The author thanks K.C. Sivakumar of the Indian Institute of Technology Madras for the many discussions on a draft version of this article. These have improved the readability, greatly.

\end{document}